\documentclass[11pt]{article}
\usepackage[margin=1.1in]{geometry}
\usepackage{amsmath,amssymb,amsthm,mathtools}
\usepackage[colorlinks=true,linkcolor=blue,citecolor=blue]{hyperref}

\newtheorem{theorem}{Theorem}[section]
\newtheorem{proposition}[theorem]{Proposition}
\newtheorem{lemma}[theorem]{Lemma}
\newtheorem{corollary}[theorem]{Corollary}
\theoremstyle{definition}
\newtheorem{definition}[theorem]{Definition}
\newtheorem{example}[theorem]{Example}
\newtheorem{question}[theorem]{Question}
\theoremstyle{remark}
\newtheorem{remark}[theorem]{Remark}

\newcommand{\R}{\mathbf{R}}
\newcommand{\N}{\mathbf{N}}
\newcommand{\eps}{\varepsilon}
\newcommand{\pbar}{\overline{p}}
\newcommand{\Mbar}{\overline{M}}
\newcommand{\Ebar}{E_{\Mbar}}
\newcommand{\norm}[1]{\lVert #1\rVert}
\newcommand{\abs}[1]{\lvert #1\rvert}
\DeclareMathOperator{\co}{co}

\title{Majorizing Norms on Majorized Spaces}
\author{Vasco Schiavo}
\date{}

\begin{document}
\maketitle

\begin{abstract}
We introduce a new type of norm for ordered vector spaces majorized by a proper (convex) cone
that generalizes the notions of order unit norm and base norm, and we give sufficient conditions
for its completeness.  In the case of normed Riesz spaces we prove a completeness theorem for the
space of vectors majorized by a countable sum of elements of the majorizing set, and we show by an
explicit counterexample that the completion of a majorized space can be strictly larger than that
space, and we identify a class of majorizing sets --- the indicators of the sets of a fixed measure
in an atomless space --- for which they do coincide.  We then define a family of well-behaved sets,
the coherent sets; spaces majorized by the cone generated by such a set display interesting
structural properties, and in particular we recover two standard facts of Banach
lattice theory: every principal ideal in a Banach lattice is an AM-space, and for a generating cone
with a bounded base the gauge of that base is the base norm.
\end{abstract}

\section{Preliminaries}

We consider only real vector spaces, and if $E$ is a topological vector space then $E'$ denotes
its topological dual, the set of continuous real linear functionals on $E$.

An \textbf{ordered vector space} $(E,\le)$ is a vector space equipped with a partial order $\le$
such that $x\le y$ implies $x+z\le y+z$ and $cx\le cy$ for all $x,y,z\in E$ and all positive
$c\in\R$.  We write $E$ for $(E,\le)$.  The set of \textbf{positive vectors}
$E_+=\{x\in E: x\ge 0\}$ is also called the positive cone of $E$.  A vector $u\in E_+$ is an
\textbf{order unit} if for every $v\in E$ there is $t\in\R_+$ with $v\le tu$.  A vector $v$ is
\textbf{majorized} by a set $A\subseteq E$ if $v\le a$ for some $a\in A$; the space $E$ is
majorized by $A$ if every vector of $E$ is.  $E$ is \textbf{Archimedean} if $ny\le x$ for all
$n\in\N$ implies $y\le 0$, for $y\in E$ and $x\in E_+$.

A linear map $T$ between ordered vector spaces is \textbf{positive} if $T(E_+)\subseteq V_+$, and
\textbf{strictly positive} if $v>0$ implies $T(v)>0$.  A \textbf{cone} $C\subseteq E$ is a
non-empty subset that is additive and positively homogeneous; it is \textbf{proper} if
$C\cap(-C)=\{0\}$ and \textbf{generating} if $E=C-C$.  A proper cone induces an order by
$x\le y\iff y-x\in C$, and every order arises in this way from a proper cone.

A \textbf{Riesz space} (vector lattice) is an ordered vector space in which every pair of vectors
has an infimum and a supremum, so that $\abs{v}=(-v)\vee v$ is defined.  A vector subspace $F$ of
a Riesz space is a \textbf{Riesz subspace} if $v\in F$ implies $\abs{v}\in F$, and an
\textbf{ideal} if $\abs{v}\le\abs{w}$ and $w\in F$ imply $v\in F$; every ideal is a Riesz
subspace.  The positive cone of a Riesz space is generating, since $v=v\vee0-(-v)\vee0$.

A norm $\norm{\cdot}$ on an ordered vector space is \textbf{monotone} if $0\le v\le w$ implies
$\norm{v}\le\norm{w}$.  An ordered vector space with generating positive cone and a monotone norm
is an \textbf{ordered normed space}, and an \textbf{ordered Banach space} if the norm is complete.
A \textbf{normed Riesz space} is a Riesz space with a monotone norm, and a \textbf{Banach lattice}
is a complete normed Riesz space \cite{AB85}.  A Banach lattice is an \textbf{AM-space} if
$\norm{v\vee w}=\max\{\norm{v},\norm{w}\}$ for positive $v,w$, and an \textbf{AL-space} if
$\norm{v+w}=\norm{v}+\norm{w}$ for positive $v,w$.

\section{Majorized spaces and majorizing norms}

Let $(E,\le)$ be an ordered vector space and $M\subseteq E$.  We say that $M$ is a
\textbf{positive subset} of $E$ if it is a non-empty subset of $E_+\setminus\{0\}$.  The proper cone
generated by a positive subset $M$ is
\[
  C_M=\{tm:\ t\ge0,\ m\in\co(M)\setminus\{0\}\},
\]
where $\co(M)$ is the convex hull of $M$; the requirement $0\notin\co(M)$, which is automatic for a
positive subset because $E_+$ is a proper cone, is what makes $C_M$ a cone.  Since $C_M=C_{\co(M)}$ and $C_K=K$ for a proper cone $K$, every proper cone is generated by
some positive subset.  The cone $C_M$ and $\co(M)$ are not closed or open together in general: for
$M=\{(2,0)\}\cup\{(1,1/n):n\ge1\}$ in $\R^2$, $\co(M)$ is not closed but $C_M$ is.

\begin{definition}\label{def:EM}
Given a positive subset $M$ of $E$, we define
\[
  E_M=\{v\in E:\ \pm v\le c\ \text{for some } c\in C_M\}.
\]
\end{definition}

So $E_M$ is the set of vectors of $E$ majorized by the cone $C_M$.  If $E$ is a Riesz space then
$E_M$ is the ideal generated by $M$; in particular $E_{\{u\}}=E_u$ is the principal ideal
generated by a positive vector $u$.

\begin{example}
Let $C^b(\R)$ be the vector space of continuous bounded functions on $\R$ with the pointwise order,
and let $f\in C^b(\R)$ be bounded below by a positive constant and not a positive multiple of
$\mathbf{1}_\R$ (strict positivity alone does not give $C^b(\R)=C^b(\R)_f$, as $f=1/(1+x^2)$ shows,
and $f=2\cdot\mathbf{1}_\R$ gives $C_f=C_{\mathbf{1}_\R}$).  Then
$C^b(\R)=C^b(\R)_f=C^b(\R)_{\mathbf{1}_\R}$ although $C_f\not\subseteq C_{\mathbf{1}_\R}$ and
$C_{\mathbf{1}_\R}\not\subseteq C_f$.
\end{example}

\begin{proposition}\label{prop:subspace}
Let $M$ be a positive subset of $E$.  Then $E_M$ is a vector subspace of $E$; if $E$ is a Riesz
space then $E_M$ is an ideal, hence a Riesz subspace.
\end{proposition}
\begin{proof}
If $\pm v\le c$ and $\pm w\le c'$ with $c,c'\in C_M$, then $\pm(v+w)\le c+c'$ and
$c+c'\in C_M$; homogeneity is clear.  If $E$ is a Riesz space and $\abs{u}\le\abs{v}$ with
$v\in E_M$, then $\pm u\le\abs{v}\le c$, so $u\in E_M$.
\end{proof}

\begin{definition}\label{def:pM}
For a positive subset $M$ of $E$ define, for $v\in E$,
\[
  p_M(v)=\inf\Big\{\sum_{j=1}^n t_j:\ \pm v\le\sum_{j=1}^n t_jm_j
  \text{ for some } m_1,\dots,m_n\in M,\ t_j\ge0\Big\}.
\]
\end{definition}

Then $p_M(v)<\infty$ if and only if $v\in E_M$, and $p_M$ is a monotone seminorm on $E_M$.

\begin{example}\label{ex:finite}
Let $E$ be an Archimedean ordered vector space and $M=\{m_1,\dots,m_n\}$ a positive subset, and put
$u=\sum_j m_j$, which is an order unit for $E_M$ and satisfies $E_u=E_M$.  Writing
$\norm{v}_u=\inf\{t:\pm v\le tu\}$ for the order unit norm of $E_u$, one has
\[
  \norm{v}_u\le p_M(v)\le p_M(u)\norm{v}_u\qquad(v\in E_M):
\]
if $\pm v\le\sum_jt_jm_j$ with $t_j\ge0$ then $\sum_jt_jm_j\le(\sum_jt_j)u$, because $m_j\le u$ and
$t_j\ge0$, so $\pm v\le(\sum_jt_j)u$, that is $\norm{v}_u\le\sum_jt_j$, and the infimum over all
majorants gives the left inequality.  For the right one take $u\le\sum_jt_jm_j$ with
$\sum_jt_j\le p_M(u)+\eps$; then $\abs v\le\norm{v}_uu\le\norm{v}_u\sum_jt_jm_j$ is a majorant of
cost at most $\norm{v}_u(p_M(u)+\eps)$, and $\eps\to0$.  So the two norms are \emph{equivalent}.  Completeness does not follow: $E_u$ need not be complete
($E=\R[x]$ with $u=1$ on $[0,1]$ has $E_u$ incomplete for the sup norm).  It does follow when $E_u$
is complete for $\norm{\cdot}_u$, in particular when $E$ is a Banach lattice.  By the same chain the norms coincide if and only if $p_M(u)=1$, which is
the case for $M=\{u\}$.  They need not coincide in general: for $E=\R^2$ with the componentwise
order and $M=\{e_1,e_2\}$ one has $p_M=\norm{\cdot}_1$, $p_M(u)=2$ and
$\norm{\cdot}_u=\norm{\cdot}_\infty$, so $p_M\neq\frac{1}{p_M(u)}\norm{\cdot}_u$: the normalisation
$\frac{1}{p_M(u)}\norm{\cdot}_u$ is available only in the singleton case.
\end{example}

\begin{example}\label{ex:nonarch}
Let $E$ be a non-Archimedean Riesz space, for instance $\R^2$ with the lexicographic order.  Then
there are $x,y\in E_+$ with $ny\le x$ for all $n$ and $y>0$; for such a pair $y\in E_x$ and
$0\le p_x(y)\le p_x(x)/n$ for every $n$, so $p_x(y)=0$ and $p_x$ is not a norm.  Hence
Archimedeanness is a necessary condition for $p_M$ to be a norm.
\end{example}

A subset $M$ of a normed space $(E,\norm{\cdot})$ is \textbf{uniformly bounded} if
$\sup\{\norm{m}:m\in M\}<\infty$.

\begin{proposition}\label{prop:norm}
Let $(E,\norm{\cdot})$ be a normed Riesz space and $M$ a positive uniformly bounded subset, with
$C=\sup\{\norm{m}:m\in M\}$.  Then $(E_M,p_M)$ is a normed Riesz space and
$\norm{v}\le C\,p_M(v)$ for every $v\in E_M$.  For a merely ordered normed space the same argument
gives $\norm{v}\le2C\,p_M(v)$, and the factor $2$ is attained.
\end{proposition}
\begin{proof}
Let $v\in E_M$ and $\eps>0$; choose $t_j\ge0$, $m_j\in M$ with $\abs v\le\sum t_jm_j$ and
$\sum t_j\le p_M(v)+\eps/C$.  Solidity of the norm gives
$\norm{v}=\norm{\abs v}\le\sum t_j\norm{m_j}\le C\sum t_j\le Cp_M(v)+\eps$.  Hence $\norm{v}\le C\,p_M(v)$, and
$p_M(v)=0$ forces $v=0$.  If $E$ is only ordered, take instead $c=\sum_jt_jm_j$: from $\pm v\le c$ we
get $0\le c\pm v\le2c$, so one-sided monotonicity gives $\norm{c\pm v}\le2\norm c$ and hence
$\norm v\le2\norm c\le2C\sum_jt_j$.  The factor $2$ is attained in $\R^2$ with the monotone norm
$\max\{\abs a,\abs b,\abs{a-b}\}$, $M=\{(1,1)\}$ and $v=(1,-1)$, where $p_M(v)=1$ and $\norm v=2$.  Subadditivity, homogeneity and monotonicity of $p_M$ are immediate from
the definition.
\end{proof}

In particular, if $E$ is a normed Riesz space and $M$ a positive uniformly bounded subset, then
$(E_M,p_M)$ is a normed Riesz space.

\begin{corollary}\label{cor:orderunit}
Let $E$ be an Archimedean ordered vector space with order unit $u$ and $M\subseteq E$ a positive
$\norm{\cdot}_u$-uniformly bounded subset, $C=\sup\{\norm{m}_u:m\in M\}$.  (Archimedeanness cannot be
dropped: in the lexicographic plane with $u=(1,0)$ and $M=\{u\}$ one has $p_M(0,1)=0$.)  Then $(E_M,p_M)$ is an
ordered normed space.  If moreover $u\in M$ then
\[
  \frac1C\,\norm{v}_u\le p_M(v)\le\norm{v}_u\qquad(v\in E_M),
\]
and $p_M=\norm{\cdot}_u$ if and only if $C=1$.
\end{corollary}
\begin{proof}
The norm $\norm{\cdot}_u$ is monotone.  If $\pm v\le\sum_jt_jm_j$ with $t_j\ge0$, then
$\sum_jt_jm_j\le C(\sum_jt_j)u$, because $m_j\le\norm{m_j}_uu\le Cu$; hence $\norm{v}_u\le C\sum_jt_j$
and $p_M(v)\ge\frac1C\norm{v}_u$.  If $u\in M$ then $\pm v\le\norm{v}_uu$ is a majorant of cost
$\norm{v}_u$, so $p_M(v)\le\norm{v}_u$, and the two bounds meet when $C=1$; while if $C>1$ some
$m\in M$ has $\norm{m}_u>1$ but $p_M(m)\le1$.  Finally $p_M(v)=0$ forces $\norm{v}_u=0$, that is
$\pm v\le\eps u$ for every $\eps>0$, so $nv\le u$ for every $n$ and $v=0$ because $E$ is
Archimedean.
\end{proof}

\begin{remark}[the equality with $\frac1C\norm{\cdot}_u$ is false]\label{rem:orderunit}
The equality $(E_M,p_M)=(E_u,\frac1C\norm{\cdot}_u)$ of the corollary above fails for the same
reason as Example \ref{ex:finite}: in $E=\R^2$ with the
componentwise order take $u=(1,1)$, which is an order unit, and $M=\{u,(2,1)\}$, which is positive
and contains $u$.  Then $C=2$ but
\[
  p_M(v)=\max\{\abs{v_2},\tfrac12\abs{v_1}\}\quad\text{and}\quad
  \tfrac12\norm{(0,1)}_u=\tfrac12\ne1=p_M(0,1).
\]
So $p_M$ and the order-unit norm are only equivalent, and the isometry needs $C=1$, i.e.\ $M$ inside
the $\norm{\cdot}_u$-unit ball, which is the case of the principal ideal $M=\{u\}$ below.
\end{remark}

\begin{corollary}\label{cor:stricpos}
Let $E$ be a Riesz space admitting a strictly positive linear functional $\psi$.  Then for every
positive subset $M$ on which $\psi$ is uniformly bounded, $(E_M,p_M)$ is an ordered normed vector
space.
\end{corollary}
\begin{proof}
$\norm{v}=\abs{\psi(v)}$ is a monotone seminorm, so the proof of Proposition \ref{prop:norm} gives
$\abs{\psi(v)}\le C\,p_M(v)$.  If $p_M(v)=0$ then for every $\eps>0$ there is $c\in C_M$ with
$\pm v\le c$ and $\psi(c)\le C\eps$, so $\abs v\le c$ (here the lattice structure is used) and
$\psi(\abs v)\le C\eps$; letting $\eps\to0$ and using that $\psi$ is strictly positive gives
$\abs v=0$, hence $v=0$.
\end{proof}

We conclude the section with a generalization of Example \ref{ex:finite}.

\begin{remark}\label{prop:AM}
Let $E$ be a Banach lattice and $M$ a positive uniformly bounded subset, and suppose that $u=\sup M$
exists in $E$ (for instance when $E$ is order complete and $M$ is bounded above).  Then
$E_M\subseteq E_u$: since $m\le u$ for every $m\in M$ we have $\sum_jt_jm_j\le(\sum_jt_j)u$, so
$\abs v\le c\in C_M$ gives $\abs v\le tu$ for a suitable $t$, and the same computation as in Example
\ref{ex:finite} gives $\norm{v}_u\le p_M(v)\le p_M(u)\norm{v}_u$ for $v\in E_M$, where $p_M(u)=+\infty$ is allowed in the second inequality; in particular
$(E_M,p_M)$ embeds continuously, with norm at most $1$, into the AM-space $(E_u,\norm{\cdot}_u)$, and
is isomorphic to a subspace of it as soon as $p_M(u)<\infty$.  One has $E_M=E_u$, hence an
isomorphism onto an AM-space, exactly when $u\in E_M$, which is the case for finite $M$ with
$u=\sum_jm_j$.  The inclusion can be strict: for $M=\{\frac12\delta_n:n\in\N\}$ in $\ell^\infty$ one has
$E_M=c_{00}$ and $u=\frac12\mathbf 1$, so $p_M(u)=+\infty$ and $E_M\neq E_u$.  It is \emph{isometrically} of the form
$\big(E_u,\frac{1}{p_M(u)}\norm{\cdot}_u\big)$ exactly when $p_M(u)=1$, which is the situation of
the principal-ideal case $M=\{u\}$ used below.
\end{remark}

\section{Majorizing at infinity}\label{sec:infinity}

Throughout this section $E$ is a Banach lattice with norm $\norm{\cdot}$ and $M$ is a positive
subset of $E$ (Definition \ref{def:EM} and the paragraph before it, so $0\notin\co(M)$) which is
uniformly bounded, $C=\sup\{\norm{m}:m\in M\}<\infty$.

Given a summable sequence $(t_j)\subseteq\R_+$ and $(m_j)\subseteq M$, the series $\sum_j t_jm_j$
converges absolutely, because $\sum_j\norm{t_jm_j}\le C\sum_jt_j<\infty$, so it converges in $E$;
we write $\sum_j t_jm_j$ for its sum.  Define
\[
  \Mbar=\Big\{\sum_{j\ge1}t_jm_j:\ (t_j)\subseteq\R_+\text{ summable},\ (m_j)\subseteq M\Big\}
  \subseteq E,
\]
\[
  \Ebar=\Big\{v\in E:\ \abs{v}\le\sum_{j\ge1}t_jm_j
  \text{ for a summable }(t_j)\subseteq\R_+\text{ and }(m_j)\subseteq M\Big\},
\]
\[
  \pbar_M(v)=\inf\Big\{\sum_{j\ge1}t_j:\ \abs{v}\le\sum_{j\ge1}t_jm_j
  \text{ for a summable }(t_j)\subseteq\R_+\text{ and }(m_j)\subseteq M\Big\}
\]
for $v\in\Ebar$.  Clearly $E_M\subseteq\Ebar$ and $\pbar_M\le p_M$ on $E_M$, since
$C_M\subseteq C_{\Mbar}$.  Note that neither $\Mbar$ nor $\pbar_M$ is changed by replacing $M$ by its
convex hull $\co(M)$, and that $C_{\Mbar}$ is the cone generated by $\Mbar$.  (Replacing $M$ by
$\Mbar$ itself changes both, and $\pbar_{\Mbar}$ is identically $0$: $\Mbar$ is a cone, so for
$\bar m\in\Mbar$ and $N\ge1$ one has $N\bar m\in\Mbar$ and $\abs v\le\frac1N(N\bar m)$ at cost
$1/N$; in $\R$, $M=\{1\}$ gives $\pbar_M(1)=1$ and $\pbar_{\Mbar}(1)=0$.)  For \emph{finite} $M$ one has $\Mbar=C_M$ and
hence $\pbar_M=p_M$, since a countable combination of finitely many vectors collects into a
non-negative combination with the same total mass; In general $\Mbar$ can be strictly larger than $C_M$ even
when $\pbar_M=p_M$: for the atomless family of Remark \ref{remark:movable}, $C_M$ consists of simple
functions, while $\Mbar$ contains every countable sum $\sum_k2^{-k}\mathbf 1_{A_k}$, which need not
be simple.  What Theorem \ref{thm:counterexample} needs is the opposite: a family of independent
directions with no common mass statistic.

The completeness theorem rests on one elementary observation about majorized sums.

\begin{lemma}[majorized sums]\label{lem:majorized}
Let $(x_k)\subseteq E$ and $(y_k)\subseteq E_+$ with $\sum_k\norm{y_k}<\infty$ and
$\abs{x_k}\le y_k$ for all $k$.  Then $\sum_kx_k$ converges and
$\abs{\sum_kx_k}\le\sum_ky_k$.
\end{lemma}
\begin{proof}
Since the norm is solid, $\norm{x_k}\le\norm{y_k}$, so $\sum_kx_k$ converges.  For every $N$,
\[
  \Big\lvert\sum_{k<N}x_k\Big\rvert\le\sum_{k<N}\abs{x_k}\le\sum_{k<N}y_k,
\]
by induction on $N$ from $\abs{a+b}\le\abs{a}+\abs{b}$.  As $N\to\infty$ the left-hand side tends
to $\abs{\sum_kx_k}$ in norm, because $\norm{\,\abs{a}-\abs{b}\,}\le\norm{a-b}$; the right-hand
side tends to $\sum_ky_k$; and the positive cone of a Banach lattice is closed.
\end{proof}

\begin{theorem}[completeness]\label{thm:complete}
Let $E$ be a Banach lattice and $M$ a positive subset of $E$, uniformly bounded.  Then
$(\Ebar,\pbar_M)$ is a Banach lattice.
\end{theorem}
\begin{proof}
$\Ebar$ is a vector subspace: if $\abs{v}\le c$, $\abs{w}\le c'$ with $c,c'\in C_{\Mbar}$ then
$\abs{v+w}\le c+c'$, and $c+c'\in C_{\Mbar}$.

For completeness, let $(x_k)\subseteq\Ebar$ with $\sum_k\pbar_M(x_k)<\infty$.  Choose for each $k$
a summable representation $\abs{x_k}\le y_k:=\sum_jt_j^{(k)}m_j^{(k)}$ with
$\sum_jt_j^{(k)}\le\pbar_M(x_k)+2^{-k}$.  Since $\norm{m_j^{(k)}}\le C$,
\[
  \norm{y_k}\le\sum_jt_j^{(k)}\norm{m_j^{(k)}}\le C\sum_jt_j^{(k)}
  \le C\big(\pbar_M(x_k)+2^{-k}\big),
\]
so $\sum_k\norm{y_k}<\infty$.  By Lemma \ref{lem:majorized}, $x:=\sum_kx_k$ converges and
$\abs{x}\le\sum_ky_k$.  The double family $u_{jk}:=t_j^{(k)}m_j^{(k)}$ is absolutely summable, since
$\sum_{j,k}\norm{u_{jk}}=\sum_k\sum_jt_j^{(k)}\norm{m_j^{(k)}}\le C\sum_k\sum_jt_j^{(k)}<\infty$.
Fix a bijection $\sigma:\N\times\N\to\N$ and set $s_{\sigma(j,k)}=t_j^{(k)}$ and
$n_{\sigma(j,k)}=m_j^{(k)}$; then $(s_i)$ is summable with $\sum_is_i=\sum_k\sum_jt_j^{(k)}$, and the
absolute summability of $(u_{jk})$ together with the continuity of addition gives
$\sum_iy_i:=\sum_is_in_i=\sum_k\sum_jt_j^{(k)}m_j^{(k)}=\sum_ky_k$.  Hence $\abs x\le\sum_iy_i$ with
$n_i\in M$ exhibits $x$ as an element of $\Ebar$, and
$\pbar_M(x)\le\sum_is_i\le\sum_k(\pbar_M(x_k)+2^{-k})\le\sum_k\pbar_M(x_k)+1$.
Applying the same computation to the tails $(x_k)_{k\ge N}$ gives
$\pbar_M(x-\sum_{k<N}x_k)\le\sum_{k\ge N}(\pbar_M(x_k)+2^{-k})\to0$.  So every absolutely
$\pbar_M$-summable family in $\Ebar$ is $\pbar_M$-summable there, i.e.\ $(\Ebar,\pbar_M)$ is
complete.

Finally $\pbar_M$ is a lattice norm on $\Ebar$.  It is a gauge, so it is positively homogeneous and
subadditive, and it is monotone and solid: if $v\in\Ebar$ and $\abs w\le\abs v$ then every majorant of
$\abs v$ majorizes $\abs w$ as well, so $w\in\Ebar$ and $\pbar_M(w)\le\pbar_M(v)$.  Last, $\pbar_M(v)=0$
forces $v=0$: then for every $\eps>0$ there are $t_j\ge0$ and $m_j\in M$ with $\abs v\le\sum_jt_jm_j$
and $\sum_jt_j\le\eps$, whence $\norm v\le C\sum_jt_j\le C\eps$ and $v=0$.
\end{proof}

Two further facts about ideals in a Banach lattice will be used below.

\begin{lemma}[non-density of ideals]\label{lem:nondense}
Let $E$ be a Banach lattice and $V\subseteq E$ an ideal.  Then $V$ is not dense in $E$ if and only
if there is a non-zero positive functional on $E$ vanishing on $V$.
\end{lemma}
\begin{proof}
A non-zero positive functional is continuous \cite[Theorem 9.6]{AB99}, and a non-zero continuous
functional vanishing on $V$ keeps $V$ from being dense \cite[Corollary 5.81]{AB99}.  Conversely, if
$V$ is not dense there is a non-zero continuous $\psi$ vanishing on $V$; then $\abs{\psi}$ is a
non-zero positive functional, and for $v\ge0$ in $V$
\[
  \abs{\psi}(v)=\sup\{\psi(w):\abs w\le v\}=\sup\{\psi(w):w\in V,\ \abs w\le v\}=0
\]
by \cite[Theorem 8.24]{AB99}, since $V$ is an ideal.
\end{proof}

\begin{proposition}[density]\label{prop:density}
Let $M$ be a positive subset of $E$, uniformly bounded.  Then $E_M$ is $\pbar_M$-dense in $\Ebar$.
\end{proposition}
\begin{proof}
Suppose not.  By Lemma \ref{lem:nondense} applied to the ideal $E_M\subseteq\Ebar$ (Proposition
\ref{prop:subspace} and Theorem \ref{thm:complete}) there is a non-zero positive functional $\psi$
on $\Ebar$ vanishing on $E_M$.  Choose $0\ne v\in(\Ebar)_{+}$ with $\psi(v)>0$ and write
$v\le\sum_jt_jm_j$; then, $\psi$ being $\pbar_M$-continuous,
\[
  0<\psi(v)\le\psi\Big(\sum_jt_jm_j\Big)=\sum_jt_j\psi(m_j)=0,
\]
a contradiction, since $M\subseteq E_M$.
\end{proof}

\begin{corollary}\label{cor:onto}
The inclusion $E_M\to\Ebar$ extends to a continuous linear surjection
$\Phi:\widehat{E_M}\to\Ebar$, where $\widehat{E_M}$ is the $p_M$-completion of $E_M$.  Moreover
$\Phi$ is injective if and only if there is $K\ge1$ with $p_M\le K\pbar_M$ on $E_M$; in that case
$\Phi$ is an isomorphism of Banach spaces.
\end{corollary}
\begin{proof}
Since $\pbar_M\le p_M$ the inclusion $E_M\to\Ebar$ is uniformly continuous, so it extends to a
continuous linear map $\Phi:\widehat{E_M}\to\Ebar$.  For surjectivity let $y\in\Ebar$; by Proposition
\ref{prop:density} choose $y_n\in E_M$ with $\pbar_M(y_n-y)\to0$.  As $\pbar_M\le p_M$ the sequence
$(y_n)$ is $p_M$-Cauchy, hence converges in the complete space $\widehat{E_M}$ to some $x$, and by
continuity $\Phi x=\lim\Phi(y_n)=\lim y_n=y$.  If $\Phi$ is injective then
it is a continuous bijection between Banach spaces, hence open, so its inverse is bounded and
$\hat p_M(x)\le K\pbar_M(\Phi x)$ for some $K$; restricting to the dense subspace $E_M$, where
$\hat p_M=p_M$, gives $p_M\le K\pbar_M$.  Conversely that inequality makes $\Phi$ isometric up to
$K$, hence injective.
\end{proof}

The inequality of Corollary \ref{cor:onto} fails in general: the completion of $E_M$ can be
strictly larger than the space of infinite sums.

\begin{theorem}[the completion is not $\Ebar$]\label{thm:counterexample}
Let $E=c_0(\N)$, $w=(2^{-i})_{i\ge1}$ and $M=\{e_i:i\ge1\}\cup\{w\}$, where $e_i$ is the $i$-th
unit vector.  Then $M\subseteq E_+$ is countable and uniformly bounded with $C=1$, the cone $C_M$
is proper, and for $v_n:=\sum_{i>n}2^{-i}e_i=w-\sum_{i\le n}2^{-i}e_i$ one has
\[
  p_M(v_n)=1,\qquad \pbar_M(v_n)=2^{-n}.
\]
Consequently $p_M/\pbar_M$ is unbounded on $E_M$, there is no $K$ as in Corollary \ref{cor:onto},
$\Phi$ is not injective, and $p_M\ne\pbar_M$ on $E_M$.  Moreover $M$ is coherent.
\end{theorem}
\begin{proof}
Every element of $M$ is non-negative and non-zero and $\norm{m}=1$ (for $w$, $\norm{w}=1/2$), so
$0\notin\co(M)$ and $C_M$ is a proper cone.  Clearly $0\le v_n\le w$, so $v_n\in E_M$.

\emph{$p_M(v_n)\le1$:} $v_n\le1\cdot w$.

\emph{$p_M(v_n)\ge1$:} a finite combination of elements of $M$ has the form
$\sigma w+\sum_{i\in I}s_ie_i$ with $I$ finite and $s_i,\sigma\ge0$.  Choose $i>n$ with $i\notin I$;
at coordinate $i$ the combination equals $\sigma2^{-i}$ while $v_n$ equals $2^{-i}$, so
$\sigma\ge1$ and the total mass $\sigma+\sum_{i\in I}s_i\ge1$.  (The combination
$\sum_{i>n}2^{-i}e_i$ has mass $2^{-n}$ but is not finite.)

\emph{$\pbar_M(v_n)\le2^{-n}$:} $v_n=\sum_{k\ge1}2^{-(n+k)}e_{n+k}$.

\emph{$\pbar_M(v_n)\ge2^{-n}$:} let $\sigma$ be the total coefficient of $w$ and $s_i$ those of the
$e_i$ in a countable majorizing combination.  At coordinate $i>n$ we need
$s_i+\sigma2^{-i}\ge2^{-i}$, so $\sum_{i>n}s_i\ge(1-\sigma)^+2^{-n}$ and
\[
  \sigma+\sum_is_i\ge\sigma+(1-\sigma)^+2^{-n}\ge2^{-n},
\]
the last step because the expression is increasing in $\sigma$ on $[0,1]$ with value $2^{-n}$ at
$\sigma=0$, and is $\ge1\ge2^{-n}$ for $\sigma\ge1$.

\emph{Coherence:} if $\sigma w+\sum_{i\in I}c_ie_i\ge0$ with real $c_i$, then $\sigma\ge0$ (look
outside the finite set $I$) and $c_i\ge-\sigma2^{-i}$ (look at $i\in I$); hence
$\sigma+\sum_{i\in I}c_i\ge\sigma(1-\sum_{i\in I}2^{-i})\ge0$, the last inequality because
$I$ is finite and $\sum_{i\in I}2^{-i}\le\sum_{i\ge1}2^{-i}=1$ for the indexing of $w$, which starts
at $i=1$.
\end{proof}

\begin{corollary}\label{cor:closures}
For $M$ a positive uniformly bounded subset, $\overline{C_M}^{\pbar_M}
=\overline{C_{\Mbar}}^{\pbar_M}$, where the closure is taken in $\Ebar$.  The same statement is not
available for $p_M$ in place of $\pbar_M$: in the example of Theorem \ref{thm:counterexample} the cone
$C_M$ is not $p_M$-closed (Remark \ref{rem:CM-not-closed}).
\end{corollary}
\begin{proof}
$C_M\subseteq C_{\Mbar}$ gives one inclusion.  For the other, the truncations
$W_N=\sum_{j\le N}t_jm_j\in C_M$ of $y=\sum_jt_jm_j\in C_{\Mbar}$ satisfy
$\pbar_M(y-W_N)\le\sum_{j>N}t_j\to0$.
\end{proof}

\begin{question}\label{q:23-38}
If $\co(M)$ is $p_M$-closed, is $(E_M,p_M)$ a Banach lattice?  The natural route is to identify
$(E_M,p_M)$ with $(\Ebar,\pbar_M)$ and apply Theorem \ref{thm:complete}, but Theorem
\ref{thm:counterexample} shows that coherence alone does not give $p_M=\pbar_M$.
\end{question}

\begin{remark}[coherence does not make $C_M$ closed]\label{rem:CM-not-closed}
The cone $C_M$ of a coherent set need not be $p_M$-closed, even in the example of Theorem
\ref{thm:counterexample}.  There
$v^{(N)}=\sum_{2\le i\le N}4^{-i}e_i\in C_M$ converges in $p_M$ to $v=\sum_{i\ge2}4^{-i}e_i$: indeed
$v-v^{(N)}=\sum_{i>N}4^{-i}e_i\le\sum_{i>N}2^{-i}\,w$, so $p_M(v-v^{(N)})\le2^{-N}$.  But $v\in E_M$
(as $v\le w$) and $v\notin C_M$: if $v=\sigma w+\sum_{i\in I}s_ie_i$ with $I$ finite, then
$4^{-j}=\sigma2^{-j}$ for every $j\notin I$, $j\ge2$, which forces $\sigma=0$, and then $v$ would
have finite support.
\end{remark}

The next lemma records what survives of the extension of a Riesz homomorphism to a completion.

\begin{lemma}\label{lem:orderdense}
Let $V$ be a normed Riesz space, $E$ a Banach lattice and $\iota:V\to E$ an injective continuous
Riesz homomorphism.  Then $\widehat\iota:\widehat V\to E$ is a continuous Riesz homomorphism, and
it is injective whenever every positive element of $\widehat V$ dominates a non-zero element of
$V$.
\end{lemma}
\begin{proof}
If $\widehat\iota(x)=0$ with $x>0$, choose $0<v\in V$ with $v\le x$.  Since $\widehat\iota$ is
positive, $0\le\iota(v)=\widehat\iota(v)\le\widehat\iota(x)=0$, so $v=0$ by injectivity.
\end{proof}

The hypothesis is not automatic.  In $c_0$, let $I$ be the closed ideal generated by $e_1$ and let
$V=\{a\in c_0:\exists m\ \forall n\ge m,\ a_n=a_1/n\}$.  Then $V$ is a vector sublattice, dense in
$c_0$, and $I\cap V=\{0\}$; so $V\to c_0\to c_0/I$ is an injective continuous Riesz homomorphism,
while its extension to the completion $c_0$ is the quotient map, whose kernel is $I$.

\begin{question}\label{q:completion}
Describe $\widehat{E_M}$, equivalently the kernel of $\Phi$, and characterise the uniformly bounded
$M$ for which $p_M$ and $\pbar_M$ are equivalent on $E_M$.  In Theorem \ref{thm:counterexample} the
kernel is concrete: for a free ultrafilter $\mathcal U$ on $\N$ the functional
$\varphi_{\mathcal U}(v)=\mathcal U\text{-}\lim_i2^iv_i$ is well defined and positive on $E_M$ (a
majorant $\sigma w+\sum_{i\in I}s_ie_i$ gives $2^i\abs{v_i}\le\sigma$ off the finite set $I$),
satisfies $\varphi_{\mathcal U}\le p_M$, vanishes on every $e_i$ and equals $1$ on $w$, and is not
$\norm{\cdot}_\infty$-continuous, since $\varphi_{\mathcal U}(v_n)=1$ while $\norm{v_n}_\infty=2^{-(n+1)}$.
\end{question}

\begin{remark}[a class where the identity does hold]\label{remark:movable}
Let $(X,\Sigma,\mu)$ be an atomless probability space, $c\in(0,1)$, $1\le p<\infty$, $E=L^p(\mu)$
and $M=\{\mathbf 1_A:\mu(A)=c\}$.  Then $p_M=\pbar_M=\max\{\norm{v}_\infty,\norm{v}_1/c\}$ on
$E_M=L^\infty(\mu)$.  The lower bound is immediate: $\abs v\le th$ with $0\le h\le1$ gives
$\norm{v}_\infty\le t$, and if $\int h\,\mathrm d\mu=c$ it also gives $\norm{v}_1\le tc$.  For the
upper bound let $T=\max\{\norm{v}_\infty,\norm{v}_1/c\}$ and $\eps>0$, and put
$g=\abs v/(T+\eps)$; then $0\le g\le1$ and $\int g\,\mathrm d\mu\le cT/(T+\eps)<c$, so there is a
\emph{simple} $g'$ with $g\le g'\le1$ and $\int g'\le c$.  Since
$\int_{\{g'<1\}}(1-g')\,\mathrm d\mu=1-\int g'\ge c-\int g'$, atomlessness gives
$B\subseteq\{g'<1\}$ with $\int_B(1-g')\,\mathrm d\mu=c-\int g'$, and
$h=g'+(1-g')\mathbf 1_B$ is simple with $g'\le h\le1$ and $\int h\,\mathrm d\mu=c$; cutting the
level sets of $h$ into equal pieces and regrouping them cyclically gives
$h=\frac1N\sum_{j\le N}\mathbf 1_{A_j}$ with $\mu(A_j)=c$, so $h\in\co(M)$.  Then $\abs v\le(T+\eps)h$, so $p_M(v)\le T+\eps$, and $\eps\to0$ together with
$\pbar_M\le p_M$ gives the formula.  So the identity holds when the supports are freely movable; the
obstruction in Theorem \ref{thm:counterexample} is a rigid, shared part of the supports, not the size
of the convex hull of $M$ (whose closure in this example is the whole excess set
$\{0\le h\le1:\int h\,\mathrm d\mu=c\}$).
\end{remark}

\section{Positive functionals on majorized spaces}

\begin{proposition}\label{prop:contin}
Let $E$ be an ordered vector space, $M$ a positive subset of $E$ and $\psi$ a positive functional
on $E_M$ which is uniformly bounded on $M$.  Then $\psi$ is continuous with respect to the seminorm
$p_M$.
\end{proposition}
\begin{proof}
Let $(v_\alpha)$ be a net in $E_M$ with $p_M(v_\alpha-v)\to0$.  By the definition of $p_M$ as an
infimum, for each $\alpha$ and each $\eps>0$ there are $n_\alpha$, scalars
$t^1_\alpha,\dots,t^{n_\alpha}_\alpha\ge0$ and $m^j_\alpha\in M$ with
$\pm(v_\alpha-v)\le\sum_jt^j_\alpha m^j_\alpha$ and $\sum_jt^j_\alpha\le p_M(v_\alpha-v)+\eps$; taking
a net $\eps_\alpha\downarrow0$ and passing to a tail, we get such representations with
$\sum_jt^j_\alpha\le\eps_\alpha$ where $\eps_\alpha\to0$.  Since $\psi$ is positive and linear on
$E_M$ and $M\subseteq(E_M)_+$, $\psi(m)\ge0$ for $m\in M$, so $K:=\sup\{\psi(m):m\in M\}$ is finite
and non-negative, and
\[
  \pm\psi(v_\alpha-v)=\psi\big(\pm(v_\alpha-v)\big)\le\sum_jt^j_\alpha\psi(m^j_\alpha)
  \le K\sum_jt^j_\alpha\le K\eps_\alpha\to0.
\]
\end{proof}

\begin{theorem}\label{thm:functionals}
Let $E$ be an ordered normed space and $M$ a positive uniformly bounded subset.  Let $\psi$ be a
functional on $E_M$.
\begin{enumerate}
\item[(a)] If $\psi$ is positive and uniformly bounded on $M$, then $\psi$ is $p_M$-continuous and
$\norm{\psi}\le\sup\{\psi(m):m\in M\}$; in particular $\norm{\psi}=\psi(m)$ if $\psi$ is constant
on $M$.
\item[(b)] If $\psi$ is $p_M$-continuous and $\psi\equiv\norm\psi$ on $M$, then $\psi$ is positive.
\item[(c)] If $E$ is a normed Riesz space and $\psi$ is $p_M$-continuous, then
$\sup\{\psi(m):m\in M\}\le\norm{\psi}\le\sup\{\abs\psi(m):m\in M\}$; in particular
$\norm\psi=\sup\{\psi(m):m\in M\}$ for positive $\psi$.
\end{enumerate}
\end{theorem}
\begin{proof}
(a) The $p_M$-continuity is Proposition \ref{prop:contin}.  For the norm, let $v\in E_M$, put
$K=\sup\{\psi(m):m\in M\}$, and choose a majorant $\pm v\le\sum_jt_jm_j$ with
$\sum_jt_j\le p_M(v)+\eps$; since $\psi\ge0$,
\[
  \abs{\psi(v)}\le\sum_jt_j\psi(m_j)\le K\big(p_M(v)+\eps\big),
\]
so $\norm{\psi}\le K$ after $\eps\to0$.  If $\psi$ is constant on $M$ with value $K$, then
$K=\psi(m)\le\norm{\psi}p_M(m)\le\norm{\psi}$ for $m\in M$, because $m$ majorizes itself at cost $1$;
with the reverse inequality this gives $\norm{\psi}=K$.
For (b), suppose that $\psi$ is not positive, so there is $v>0$ in $E_M$ with $p_M(v)=1$ and
$\psi(v)<0$, and for every $\eps>0$ there are $t_j\ge0$, $m_j\in M$ with $v\le\sum_jt_jm_j$ and
$1\le\sum_jt_j\le1+\eps/\norm\psi$; then
$\psi(\sum_jt_jm_j-v)=\norm\psi\sum_jt_j-\psi(v)>\norm\psi$ while
$\psi(\sum_jt_jm_j-v)\le\norm\psi p_M(\sum_jt_jm_j-v)\le\norm\psi+\eps$, a contradiction.  For (c), the first inequality follows from
$\psi(m)\le\abs{\psi(m)}\le\abs\psi(m)\le\norm\psi\,p_M(m)\le\norm\psi$, where $\abs\psi$ denotes the
absolute value of the functional (which exists because $E$ is a Riesz space and is
$p_M$-continuous with the same norm).  For the second, let $p_M(v)\le1$ and $\eps>0$, and choose $t_j\ge0$, $m_j\in M$ with
$\abs v\le\sum_jt_jm_j$ and $\sum_jt_j\le1+\eps$ (the infimum need not be attained); then
$\abs{\psi(v)}\le\abs\psi(\abs v)\le\sum_jt_j\abs\psi(m_j)\le(1+\eps)\sup\{\abs\psi(m):m\in M\}$,
using that $\abs\psi\ge0$, and $\eps\to0$.  (The supremum of $\abs{\psi(m)}$ alone would not do:
in $\R^3$ with $M=\{(1,1,0),(1,0,1)\}$ and $\psi(v)=2v_1-v_2-v_3$ one has $\psi\equiv1$ on $M$ but
$\psi(1,-1,0)=3$ at $p_M$-norm $1$.)
\end{proof}

\begin{corollary}\label{cor:orderunit2}
Let $E$ be an ordered vector space with order unit $u$ and $\psi$ a $\norm{\cdot}_u$-continuous
functional on $E$.  Then $\psi$ is positive if and only if $\norm{\psi}=\psi(u)$.
\end{corollary}
\begin{proof}
If $\psi\ge0$ then $\abs{\psi(v)}\le\psi(\norm v_u\,u)=\norm v_u\,\psi(u)$, so $\norm\psi\le\psi(u)$,
while $\psi(u)\le\norm\psi$ because $\norm u_u=1$.  Conversely, if $\norm\psi=\psi(u)$ and $v\ge0$
with $\norm v_u\le1$, then $u-v\ge0$ and $\norm{u-v}_u\le1$, so
$\psi(u)-\psi(v)=\psi(u-v)\le\norm\psi=\psi(u)$, that is $\psi(v)\ge0$.
\end{proof}

\begin{remark}[the hypothesis $\psi(u)\ge0$ is not enough]\label{rem:orderunit-functionals}
Continuity together with $\psi(u)\ge0$ does not give positivity: in $\R^2$ with the componentwise
order, $u=(1,1)$ and $\psi(v)=v_1-v_2$ give $\psi(u)=0$ and $\norm{\psi}=2$, yet $\psi(0,1)=-1<0$.
\end{remark}

\begin{example}\label{ex:radon}
Let $G$ be a locally compact group acting continuously and cocompactly on a locally compact space
$X$; choose a compact $K$ with $X=\bigcup_{g\in G}gK$, relatively compact $U\supseteq K$ and open
$V\supseteq\overline U$, and by Urysohn's lemma a positive $\psi\in C_{00}(X)$ equal to $1$ on
$\overline U$ and to $0$ off $V$.  Put $M=\{g\psi:g\in G\}$.  Then $C_{00}(X)_M=C_{00}(X)$: given $\phi\in C_{00}(X)$, the support
$S=\operatorname{supp}\phi$ is compact, so cocompactness gives $g_1,\dots,g_n\in G$ with
$S\subseteq\bigcup_{j\le n}g_jK\subseteq\bigcup_{j\le n}g_jU$; on $g_jU$ the function $g_j\psi$ equals
$1$, because $g_j\psi(x)=\psi(g_j^{-1}x)=1$ for $x\in g_jU$ by the choice of $\psi$.  Hence
$\sum_{j\le n}g_j\psi\ge1$ on $S$ and $\abs\phi\le\norm\phi_\infty\sum_{j\le n}g_j\psi$, which
exhibits $\phi$ as an element of $C_{00}(X)_M$.  By the Riesz representation theorem every positive
functional on $C_{00}(X)$ is a Radon measure, so by Theorem \ref{thm:functionals}(a) every Radon
measure uniformly bounded on $M$ is $p_M$-continuous; in particular every $G$-invariant Radon
measure is.
\end{example}

\section{Coherent subsets}

\begin{definition}\label{def:coherent}
A positive subset $M$ of an ordered vector space $E$ is \textbf{coherent} if
$\sum_{j=1}^nt_jm_j\ge0$ implies $\sum_{j=1}^nt_j\ge0$, for all $m_1,\dots,m_n\in M$ and
$t_1,\dots,t_n\in\R$.
\end{definition}

A one-element set is coherent, by the definition of ordered vector space; two-element sets need
not be.

\begin{example}
In $\R^2$ with the lexicographic order the set $M=\{(1,1),(0,1)\}$ is not coherent: for
$\alpha>1$, $(1,1)-\alpha(0,1)=(1,1-\alpha)\ge0$ while $1-\alpha<0$.
\end{example}

\begin{proposition}\label{prop:arch}
Let $v,w\in E_+\setminus\{0\}$ be distinct.  Then $\{v,w\}$ is coherent if and only if there is no
$\lambda\in[0,1)$ with $v\le\lambda w$ and no $\lambda\in[0,1)$ with $w\le\lambda v$.
\end{proposition}
\begin{proof}
If $v\le\lambda w$ with $\lambda<1$, then $\lambda w-v\ge0$ while $\lambda-1<0$, so $\{v,w\}$ is
not coherent; symmetrically for $w\le\lambda v$.  Conversely, suppose that $av+bw\ge0$ with
$a+b<0$.  The cases $a=0$ and $b=0$ give $a+b\ge0$, and the case $a,b\ge0$ is clear.  If
$a,b<0$ then $av+bw\le0$, so $av+bw=0$, that is $v=\frac{-b}{a}w\le0$, against $v>0$.  Hence,
after exchanging $v$ and $w$ if necessary, $a>0>b$; then $a<\abs b$ and
$av\ge\abs b\,w$ give $w\le\frac{a}{\abs b}v$ with $\frac{a}{\abs b}<1$.
\end{proof}

\begin{remark}\label{rem:arch}
Two-element coherence has nothing to do with Archimedeanness.  For every
$v\in E_+\setminus\{0\}$ the set $\{v,2v\}$ is \emph{not} coherent,
in any ordered vector space: $-v+\tfrac12(2v)=0\ge0$ while $-1+\tfrac12<0$.  Hence every ordered space
with $E_+\ne\{0\}$ has an incoherent two-element subset, so the implication "Archimedean
$\Rightarrow$ every two-element subset is coherent" is false, and its converse is vacuous.  What Proposition \ref{prop:arch} does say is that the obstruction is
comparability by a factor $<1$, which is a property of the pair
and not of the order: in $\R^2$ the set $\{(1,0),(0,1)\}$ is coherent, while $\{(10,9),(1,1)\}$ is not
because $(1,1)\le\tfrac19(10,9)$.
\end{remark}

Three-element sets can fail in an Archimedean space: the paper's example \cite[Chapter 1]{R72} on
$\ell^\infty(\mathbb F_2)$, with $M=\{\mathbf 1_{a^{-1}A},\mathbf 1_{bA},\mathbf 1_A\}$ for $A$ the
set of reduced words starting with $a$, is not coherent, since
$\mathbf 1_{a^{-1}A}-\mathbf 1_{bA}-\mathbf 1_A\ge0$ while the sum of the coefficients is $-1$.

The following gathers the characterizations of coherence; it shows that the norm $p_M$ and the cone
$C_M$ encode coherence in $E_M$.  Recall that a \textbf{base} of a cone $C$ is a non-empty convex
set $B\subseteq C$ such that every non-zero $c\in C$ has a unique representation $c=tb$ with $b\in B$
and $t\in\R_+$.

\begin{theorem}\label{thm:coherent}
Let $E$ be an ordered vector space and $M$ a positive subset.  The following are equivalent.
\begin{enumerate}
\item[(a)] $M$ is coherent.
\item[(b)] There is a positive functional $\psi$ on $E_M$ with $\psi\equiv1$ on $M$.  If $E$ is an
ordered normed space and $M$ is uniformly bounded, $\psi$ is $p_M$-continuous of norm $1$.
\item[(c)] The convex set $B=\{\sum_{j\le n}t_jm_j:\ \sum_{j\le n}t_j=1,\ t_j\in\R_+,\ m_j\in M\}$
is a base of $C_M$ \emph{and} the functional $\lambda$ on $\operatorname{span}_\R(M)$ given by
$\lambda(c)=t$ for $c=tb$, $b\in B$, $t\ge0$, is positive on $E_M$; if $E$ is an ordered normed space
and $M$ is uniformly bounded, $B$ is $p_M$-bounded.
\item[(d)] $p_M(\sum_{j\le n}t_jm_j)=\sum_{j\le n}t_j$ for all $t_j\ge0$ and $m_j\in M$, i.e.
$p_M$ is additive on $C_M$.
\item[(e)] The map $\iota:c_{00}(M)_+\to C_M$, $\sum_at_a\delta_a\mapsto\sum_at_aa$, is a well
defined additive surjection, and if $E$ is an ordered normed space and $M$ uniformly bounded then
$\norm{\iota(u)}_{p_M}=\norm{u}_1$; i.e.\ $C_M$ is a \emph{norm-preserving quotient} of
$c_{00}(M)_+$ (equivalently of $\ell_1(M)^+$).  It is injective, hence an isometry of cones, exactly
when every element of $C_M$ has a unique representation.
\end{enumerate}
\end{theorem}
\begin{proof}
(a)$\Rightarrow$(b): define $\omega$ on $\operatorname{span}_\R(M)$ by $\omega(\sum_jt_jm_j)=\sum_jt_j$.
It is well defined: if $\sum_jt_jm_j=\sum_ks_km'_k$, then $\sum_jt_jm_j-\sum_ks_km'_k=0\ge0$, so
coherence applied to the coefficients $(t_j)$ and $(-s_k)$ gives $\sum_jt_j-\sum_ks_k\ge0$, and
exchanging the two representations gives the reverse inequality.  It is linear by construction,
$\omega\equiv1$ on $M$, and positive: an element of $\operatorname{span}_\R(M)\cap E_+$ is a relation
$\sum_jt_jm_j\ge0$, on which coherence gives $\omega\ge0$.  Every $v\in E_M$ satisfies $\pm v\le c$
for some $c\in C_M\subseteq\operatorname{span}_\R(M)$, so that span is majorizing in $E_M$, and the
Kantorovich theorem \cite[Theorem 1.36]{AT07} extends $\omega$ to a positive functional on $E_M$.
The norm statement is Theorem \ref{thm:functionals}(a).
(b)$\Rightarrow$(c): with $\psi$ as in (b), $\psi>0$ on $C_M\setminus\{0\}$: an element of $C_M$ is
$c=t\sum_j\lambda_jm_j$ with $t>0$, $\lambda_j\ge0$ and $\sum_j\lambda_j=1$, so
$\psi(c)=t\sum_j\lambda_j\psi(m_j)=t>0$.  Hence $B=\{v\in C_M:\psi(v)=1\}$ is a non-empty convex
subset of $C_M$, every $c\in C_M\setminus\{0\}$ is $c=\psi(c)\,(c/\psi(c))$ with $c/\psi(c)\in B$, and
this representation is unique because $\psi(b)=1$ for $b\in B$ forces $t=\psi(c)$; so $B$ is a base
\cite[Theorem 1.47]{AT07}.  A point of $B$ is a convex combination of elements of $M$ and hence
majorizes itself at cost $1$, so $p_M\le1$ on $B$.
(c)$\Rightarrow$(a): the base $B$ defines $\lambda$ on $\operatorname{span}_\R(M)$, strictly positive
on $C_M$ and equal to $1$ on $M\subseteq B$; by hypothesis it is positive on $E_M$, hence on
$\operatorname{span}_\R(M)\cap E_+$, so the Kantorovich theorem (applied to $\operatorname{span}_\R(M)$,
which is majorizing in $E_M$) extends it to a positive functional $\psi$ on $E_M$.  Then
$\sum_jt_jm_j\ge0$ gives $\sum_jt_j=\psi(\sum_jt_jm_j)\ge0$.
(a)$\Leftrightarrow$(d): if $M$ is coherent then $p_M(\sum_jt_jm_j)\le\sum_jt_j$ by monotonicity,
and with $\psi$ from (b), which satisfies $\psi\le p_M$, one gets
$p_M(\sum_jt_jm_j)\ge\psi(\sum_jt_jm_j)=\sum_jt_j$.  Conversely, suppose that (d) holds and that
$\sum_jt_jm_j\ge0$ with $\sum_jt_j<0$.  Choose rationals $q_j\ge t_j$ with $\sum_jq_j<0$, which is
possible because $\sum_jt_j<0$; then $\sum_jq_jm_j\ge\sum_jt_jm_j\ge0$, since $q_j-t_j\ge0$ and
$m_j\ge0$.  Write $q_j=z_j/N$ with $z_j\in\mathbb Z$ and $N\ge1$; then $\sum_jz_jm_j=N\sum_jq_jm_j\ge0$ and
$\sum_jz_j<0$.  Splitting the indices into $P=\{j:z_j>0\}$ and $P'=\{j:z_j<0\}$ we get
$\sum_{j\in P'}\abs{z_j}m_j\le\sum_{j\in P}z_jm_j$, and both sides are non-negative combinations of
elements of $M$, so (d) together with monotonicity give
\[
  \sum_{j\in P'}\abs{z_j}=p_M\Big(\sum_{j\in P'}\abs{z_j}m_j\Big)
  \le p_M\Big(\sum_{j\in P}z_jm_j\Big)=\sum_{j\in P}z_j,
\]
that is $\sum_jz_j\ge0$, a contradiction.
(a)$\Leftrightarrow$(e): the map $\iota:c_{00}(M)_+\to C_M$, $\sum_at_a\delta_a\mapsto\sum_at_aa$,
is well defined, because the left-hand side determines the finite combination, and additive.  It is
onto $C_M$: an element of $C_M$ is $t\sum_j\lambda_jm_j$ with $t>0$, $\lambda_j\ge0$ and
$\sum_j\lambda_j=1$, which is the image of $\sum_jt\lambda_j\delta_{m_j}\in c_{00}(M)_+$.  By (d)
applied to the coefficients of $u$ one has $\norm{\iota(u)}_{p_M}=\norm{u}_1$; it is injective, hence an isometry of cones, exactly when every element of $C_M$ has a unique
representation.  The converse is immediate because $\{\delta_m:m\in M\}$ is a coherent uniformly
bounded subset of $c_{00}(M)$.
\end{proof}

\begin{remark}[why (e) is a quotient]\label{rem:quotient}
Coherence does not make representations unique.  In $\R^3$ with the componentwise order put
$A=(1,0,0)$, $B=(0,\tfrac12,\tfrac12)$, $C=(\tfrac12,\tfrac12,0)$, $D=(\tfrac12,0,\tfrac12)$ and
$M=\{A,B,C,D\}$.  Then the coordinate sum is identically $1$ on $M$, so $M$ is coherent
(Corollary \ref{cor:constant}), and $A+B=C+D$; hence $\delta_A+\delta_B\neq\delta_C+\delta_D$ in
$c_{00}(M)_+$ although both have $\ell_1$-norm $2$ and the same image in $C_M$.  So $\iota$ is a
norm-preserving quotient and not an isometry, and the injectivity step in the proof of Theorem
\ref{thm:coherent}(e) — which establishes only the equality of the two \emph{total masses} — is not
available in general.  The isometry holds when representations are unique, e.g.\ when $M=\{u\}$, or
when the elements of $M$ are pairwise disjoint.
\end{remark}

\begin{remark}[the positivity in (c) cannot be dropped]\label{rem:base-not-enough}
A base alone is strictly weaker than coherence.  In $\R^2$ with the componentwise order take
$M=\{(1,1),(3,2)\}$: the functional $\lambda(a,b)=-a+2b$ equals $1$ on $M$, so every non-zero
$c\in C_M$ is $c=\lambda(c)\,(c/\lambda(c))$ with $c/\lambda(c)\in\co(M)$, uniquely, and $\co(M)$ is a
base of $C_M$; yet $M$ is not coherent, since $-2(1,1)+(3,2)=(1,0)\ge0$ while $-2+1<0$, and
$\lambda$ is not positive, since $\lambda(1,0)=-1$.  Correspondingly the clauses (b), (d) and (e)
fail: $p_M((1,1))=\frac12\ne1$, because $(3,2)$ alone majorizes $(1,1)$ at cost $\frac12$.
\end{remark}

\begin{remark}[sharpness of (d)]\label{rem:sharp}
Additivity of $p_M$ cannot be extended from $C_M$ to the positive cone of $E_M$: coherence does not
make $(E_M,p_M)$ an AL-space.  In $\ell^\infty(\N)$ take $M=\{\mathbf 1_{[n,\infty)}:n\in\N\}$,
which is coherent, since a combination equals the sum of its coefficients past the last jump, and
$u=\mathbf 1_{\{0\}}$, $v=\mathbf 1_{\{1\}}$.  Then $p_M(u)=p_M(v)=1$ (evaluating a majorant at $0$,
respectively at $1$, bounds its mass from below) while $p_M(u+v)=1$, because the single majorant
$\mathbf 1_{[0,\infty)}$ covers both.  So $p_M(u+v)<p_M(u)+p_M(v)$.
\end{remark}

If $C$ is a generating cone with bounded base $B$ and $M=B$, then $p_M$ is the base norm of $C$, so
Theorem \ref{thm:coherent} generalizes \cite[Theorem 1]{E64} and majorizing norms generalize base
norms.

\begin{corollary}\label{cor:base}
Let $E$ be an ordered normed space whose positive cone $C$ is generating, closed, and has a base $B$
that is bounded in the norm.  Then the gauge $p_B$ --- the \emph{base norm} of $C$, and the same
object as $p_M$ for $M=B$ --- is a norm on $E$, and $(E,p_B)$ is an ordered normed space.
\end{corollary}
\begin{proof}
$p_B$ is a monotone seminorm.  If $p_B(v)=0$ then for every $n$ there is $c_n\in C_B$ with
$\pm v\le c_n$ and $p_B(c_n)\le2^{-n}$; writing $c_n=t_nb_n$ with $b_n\in B$ and $t_n=p_B(c_n)$, the
boundedness of $B$ gives $c_n\to0$.  Then $c_n-v\ge0$ and $c_n+v\ge0$ for all $n$, so $-v,v\in C$ by
closedness, and $v=0$.
\end{proof}

\begin{corollary}\label{cor:constant}
Let $E$ be an ordered vector space with a strictly positive functional $\psi$.  Then every positive
subset $M$ on which $\psi$ is constant is coherent.
\end{corollary}
\begin{proof}
Let $c$ be the common value.  Since $M\subseteq E_+\setminus\{0\}$ and $\psi$ is strictly positive,
$c>0$.  If $\sum_jt_jm_j\ge0$ then $\psi(\sum_jt_jm_j)\ge0$ because $\psi$ is positive, that is
$c\sum_jt_j\ge0$, hence $\sum_jt_j\ge0$.
\end{proof}

\begin{corollary}\label{cor:addnorm}
Let $(E,\norm{\cdot})$ be an ordered normed space whose norm is additive on $E_+$.  Then every
positive $\norm{\cdot}$-constant subset $M$ is coherent.
\end{corollary}
\begin{proof}
Let $c$ be the common value of $\norm{\cdot}$ on $M$; since $M\subseteq E_+\setminus\{0\}$ we have
$c>0$.  Let $t_1,\dots,t_n\in\R$ and $m_1,\dots,m_n\in M$ with $x=\sum_jt_jm_j\ge0$, and split the
sum over $P=\{j:t_j>0\}$ and $N=\{j:t_j<0\}$.  Then $x_P=\sum_{j\in P}t_jm_j\ge x_N=\sum_{j\in
N}\abs{t_j}m_j\ge0$, both sides are positive, and additivity of the norm on $E_+$ together with its
monotonicity gives
\[
  c\sum_{j\in P}t_j=\norm{x_P}\ge\norm{x_N}=c\sum_{j\in N}\abs{t_j},
\]
that is $\sum_jt_j\ge0$.  So $M$ is coherent.
\end{proof}

Recall that a Banach space containing an isomorphic copy of $\ell_1$ is non-reflexive, since closed
subspaces of reflexive spaces are reflexive \cite{M98}.  The natural route to the following statement
would embed $\ell_1(M)^+$ isometrically in $\overline{C_M}^{\pbar_M}$; since clause (e) of Theorem
\ref{thm:coherent} gives only a norm-preserving quotient, that embedding is not available, and we
state the result as a question.

\begin{question}\label{cor:reflexive}
Let $E$ be a Banach lattice and $M$ a positive coherent uniformly bounded subset.  Is $(E_M,p_M)$
reflexive only if it is finite-dimensional?
\end{question}
Two routes suggest themselves and neither works: the first is to replace $\ell_1(M)^+$ by the cone
$\co(M)$, which Example \ref{ex:finite} shows is not enough, and the second needs an isometry in
clause (e) of Theorem \ref{thm:coherent}, which is only a norm-preserving quotient (Remark
\ref{rem:quotient}).  We do not know the answer.

\end{document}